\documentclass[final]{siamltex}
\usepackage{latexsym,amsmath,amssymb}
\usepackage{amsfonts}
\usepackage{comment}
\usepackage{graphicx}
\usepackage{epsfig}
\usepackage{tabularx}
\usepackage{multirow}
\usepackage{color}
\usepackage{chngpage}
\usepackage{algorithm}
\usepackage{algorithmic}

\newtheorem{Theo}{Theorem}

\newtheorem{Rem}{Remark}
\newtheorem{Lem}{Lemma}

\numberwithin{equation}{section}
 \numberwithin{Lem}{section}
 \numberwithin{Defi}{section}
 \numberwithin{Theo}{section}
 \numberwithin{Rem}{section}
  \numberwithin{Coro}{section}
  \numberwithin{Fig}{section}

\def\R{\mathbb{R}}

\allowdisplaybreaks
\title{Sharp pointwise-in-time error estimate of L1 scheme for nonlinear subdiffusion equations}
\author{
Dongfang Li\thanks{School of Mathematics and Statistics, Huazhong University of Science and
Technology, Wuhan, 430074, China.
  ({\tt dfli@mail.hust.edu.cn}).
  The research of this author was supported by the National Natural Science Foundation of China under grants No.
11771162. }
\and Hongyu Qin\thanks{School of Mathematics and Statistics, Wuhan University, Wuhan, 430072, China.
  ({\tt qinhongyuhust@sina.com}).
  }
\and Jiwei Zhang\thanks{School of Mathematics and Statistics, Wuhan University, Wuhan, 430072, China.
  ({\tt jiweizhang@whu.edu.cn}).
  The research of this author was supported by the National Natural Science Foundation of China under grants No.
11771035. }
}

\begin{document}

\maketitle

\begin{abstract} An essential feature of the subdiffusion equations with the $\alpha$-order time fractional derivative is the weak singularity at the initial time.  The weak regularity of the solution is usually characterized by a regularity parameter $\sigma\in (0,1)\cup(1,2)$. Under this general regularity assumption,  we here obtain the pointwise-in-time error estimate  of the widely used L1 scheme for nonlinear subdiffusion equations. To the end, we present a refined discrete fractional-type Gr\"onwall inequality and a rigorous analysis for the truncation errors.  Numerical experiments  are provided to demonstrate the effectiveness of our theoretical analysis.

\end{abstract}


\begin{AMS}
{65M06,65M12,65M15,35R11}
\end{AMS}

\pagestyle{myheadings}
\thispagestyle{plain}
\markboth{D. Li, H. Qin and J. Zhang}
{Sharp pointwise error estimate}

\section{Introduction}\label{int}
In this paper, we consider sharp pointwise-in-time
 error estimate of  L1 scheme in time for solving
the following nonlinear subdiffusion equations
\begin{align}
~\partial_t^\alpha u-\Delta u=f(u),
\quad
x\in\Omega\times(0,T]
\label{i1}
\end{align}
with the initial and boundary conditions
\begin{align}
\begin{array}{ll}
u(x,0)=u_0(x), & x\in\Omega, \\
u(x,t)=0, & x\in\partial\Omega\times[0,T],
\end{array}
\label{i1-ib}
\end{align}
where $\Omega=(0,L)^d\subset \R^d$ ($d\geq 1$).
The time fractional
Caputo derivative is defined as
\begin{equation} \label{defin-1}
 ~\partial_t^\alpha u(x,t)
=\frac{1}{\Gamma(1-\alpha)}\int_0^t\frac{\partial
u(x,s)}{\partial s}\frac{1}{(t-s)^\alpha}ds, \quad 0<\alpha <1.
\end{equation}
Here $\Gamma(\cdot)$ denotes the Gamma function. The equations provide a powerful tool to describe anomalous diffusion in different physical situations, see e.g., \cite{gor02,hen00,liao20}. Hence, the theoretical
and numerical analysis of the models have attracted the interest of plenty of researchers.

In developing numerical methods for solving the subdiffusion
problem~\eqref{i1}, an important consideration is that the
solution~$u$ is typically less regular than in the case of a classical
parabolic PDE (as the limiting case~$\alpha\to1$).
 For instance, Jin et al. \cite{jin18} show that if the initial condition $u_0\in H_0^1(\Omega)\cap H^2(\Omega)$, the solution to problem \eqref{i1} satisfies
$\|\partial_tu(t)\|_{L^2(\Omega)}\le Ct^{\alpha-1}$.
Maskari and Karaa \cite{mas19} obtain that if $u_0\in \dot{H}^\nu(\Omega)$ with $\nu \in(0,2]$, the solution of problem \eqref{i1} satisfies
$
\|\partial_tu(t)\|_{L^2(\Omega)}\le Ct^{\nu\alpha/2-1}$, which implies that there exists a parameter $\sigma\in (0, \alpha]$ and $u_t\rightarrow\infty$ as
$t\rightarrow0^+$. One can refer to more works \cite{kop,cao16,jin-la16,li183,li191,mcl19,mcl20} on the discussion of the regularity of solutions. With loss of generality,  we assume the solution regularity satisfies
\begin{equation}\label{assm-1}
\|{\partial^m_t} u\|_{L^2(\Omega)}\leq Ct^{\sigma-m},\quad  \text{ for } m = 1,2, \text{ and } \sigma\in (0,1)\cup(1,2].
\end{equation}


Under the regularity assumption $\sigma = \alpha$ in \eqref{assm-1}, many works indicate that the convergence order with the maximum norm in time is $\mathcal{O}(\tau^\alpha)$, where $\tau$ is the temporal stepsize.
Here we refer readers to \cite{liao-li,sty17,lia19,kop21} for the L1 and L2-type schemes on the uniform meshes, \cite{jin18,lub86} for the the convolution
quadrature (CQ) Euler method and  \cite{jin18,jin19} for the CQ BDF methods.
In addition, numerical simulations show an interesting phenomenon that  the convergence order of the L1 scheme is $\mathcal{O}(\tau^\alpha)$ as $t$ tends to $0$, and  $\mathcal{O}(\tau)$  at the final time $t=T$.
It motivates much works to consider the pointwise error estimate. For linear subdiffusion equations (i.e., $f(u)=0$),
Gracia et al. \cite{gra18} proved the temporal error of L1 scheme is of $\tau t_n^{\alpha-1}$.
 Yan et el.  \cite{yan18} considered time-stepping error estimates of the modified L1 scheme.
Jin et al. \cite{jin16} showed if the initial condition $u_0(x)\in L_2(\Omega)$, the temporal error of L1 scheme is of $\tau t_n^{-1}$. After that,
they further obtained \cite{jin17}  time-stepping error estimates of some high-order BDF convolution quadrature methods.  Mustapha and McLean  \cite{mus15,mcl15} investigated  time-stepping error bounds of discontinuous Galerkin methods for fractional diffusion problems.  For the nonlinear subdiffusion equations\eqref{i1},
 Maskari and Karaa \cite{mas19} study the optimal pointwise-in-time error estimates based on the CQ Euler method.
As far as we know, it still remains open to obtain the pointwise error estimate of L1 scheme for nonlinear subdiffsuion equations.

The main goal of this paper is to present the sharp pointwise-in-time error estimate of the widely used L1 scheme for the nonlinear subdiffusion equations under the regularity assumption \eqref{assm-1}. Generally, our goals are theoretically challenging mainly due to two reasons. On the one hand, the numerical Caputo formula
always has a form of discrete convolutional summation. The temporal truncation error varies at each time level and has a close relation with the time level $n$, which leads to the consistency analysis  becomes too cumbersome to implement in practice. On the other hand,  due to the nonlinearity, the evolutions of the solutions and errors are non-monotone decreasing, which requires the rigorous and refined analysis step by step.

In this paper, we consider the L1 scheme on the uniform meshes to approximate the time fractional derivative and  central finite difference scheme with the stepsize $h$ to discretize the diffusion term.
We overcome the mentioned difficulties and paint a full
picture for sharp pointwise error estimate of the L1 scheme for the nonlinear time
fractional parabolic problems under assumption \eqref{assm-1}.  The given results indicate that
\begin{itemize}
\item Suppose $\sigma\in (0,\alpha)$, one has
   \begin{eqnarray*}
\|u^{n}-U^n\|_{\infty} \lesssim
\tau^{\sigma+1-\alpha} t_n^{\alpha-1}+ t_n^\alpha h^2, n=1,2,\cdots,N,
\end{eqnarray*}
where $u^{n}$ and $U^n$ are theoretical and numerical solutions.
The result has never been found before, even for the linear time fractional problems.

\item Suppose that  $\sigma=\alpha$, one has
   \begin{eqnarray*}
\|u^{n}-U^n\|_{\infty} \lesssim
\tau t_n^{\alpha-1}+t_n^\alpha h^2, n=1,2,\cdots,N.
\end{eqnarray*}
The pointwise error estimates are firstly given for the nonlinear problems. Especially, when $f(u)=0$, problem \eqref{i1} is reduced to the linear subdiffusion models. Then, the given result agrees with the one in \cite{gra18}.
\item  Suppose the solution is smoother (i.e., $\sigma>\alpha$),  one has
\begin{eqnarray*}
\|u^{n}-U^n\|_{\infty} \lesssim
\begin{cases}
\tau^{\sigma+1-\alpha} t_n^{\alpha-1}+ t_n^\alpha h^2,~~~~~~&\alpha<\sigma<1,\\
\tau^{2-\alpha} t_n^{\alpha+\sigma-2}+ t_n^\alpha h^2,~~~~~~&1<\sigma\leq 2,
\end{cases} n=1,2,\cdots,N.
\end{eqnarray*}
From the results, one can see that when $t_n\rightarrow 0$, the error estimate is of $\tau^\sigma$.  When $t_n$ is far away from $0$, different convergence results can be found.
Especially, the maximum error in the whole domain is of $\tau^{\min(\sigma, 2-\alpha)}$.  The results on the maximum errors  agree with  the error estimates in \cite{liao-li}.
\end{itemize}

The rest of the paper is organized as follows. In Section \ref{sec2}, we present the fully-discrete scheme and the time-stepping error estimates. In Section \ref{sec3}, we present the discrete fractional Gr\"{o}nwall inequalities and a detailed proof of the main results. In Section \ref{sec4}, we give several numerical examples to confirm the theoretical results. Finally, we have the conclusions and discussions in Section \ref{sec5}.

Throughout the paper, notation $ A\lesssim B$ implies
that there exists a constant $c$ such that $A\leq cB$.

\section{Fully discrete scheme and main convergence results}\label{sec2} In this section, we present the fully-discrete scheme and the main convergence results.

\subsection{Fully discrete scheme}
In this section, we will present the fully-discrete scheme and the main convergence results. In this study, we always assume that $u(\cdot,t)\in C^4(\Omega)$ for every $t$ so that the central finite difference method is available to approximate the diffusion term.

 Let $\mathcal{T}_\tau = \{t_n | t_n=n\tau; \; 0\leq n\leq N \}$ be a uniform
partition of $[0,T]$ with the time step $\tau=T/N$.
Let $h=L/M$ be the spatial stepsize with $M$ a give positive integer. Denote $u^n_{i_1,i_2,\cdots,i_d}=u(i_1h, i_2 h,\cdots,i_d h, t_n)$, where $i_j=1,\cdots,M-1, \;j=1,\cdots,d$ and $n=0,1,\cdots,N$.
At the grid point $(i_1h, i_2h,\cdots,i_d h)$ , the central finite difference method is given by
\begin{eqnarray} \label{central}
\sum_{j=1}^du_{x_{j}x_j}^n&=&\frac{1}{h^2}\sum_{j=1}^d (u^n_{i_1,\cdots, x_j-1,\cdots,i_d}-2 u^n_{i_1,i_2,\cdots,i_d}+u^n_{i_1,\cdots, x_j+1,\cdots,i_d}    )+R^n_{i_1,i_2,\cdots,i_d}\nonumber\\
&:=&\sum_{j=1}^d \delta_{x_j}^2 u^n_{i_1,i_2,\cdots,i_d}+R^n_{i_1,i_2,\cdots,i_d},
\end{eqnarray}
where $R^n_{i_1,i_2,\cdots,i_d}=\mathcal{O}(h^2)$ is the spatial truncation error.

The standard $L1$-approximation to the time fractional derivative is given by
\begin{eqnarray} \label{l1sch}
\partial_{t_n}^\alpha u&=&
\frac{1}{\Gamma(1-\alpha)}\sum_{j=1}^{n}\frac{u^j_{i_1,i_2,\cdots,i_d}-u^{j-1}_{i_1,i_2,\cdots,i_d}}{\tau}
\int_{t_{j-1}}^{t_{j}}\frac{1}{(t_n-s)^\alpha}ds+ r^n_{i_1,i_2,\cdots,i_d}\nonumber\\
&=&
\frac{\tau^{-\alpha}}{\Gamma(2-\alpha)}\sum_{j=1}^na_{n-j}(u^j_{i_1,i_2,\cdots,i_d}-u^{j-1}_{i_1,i_2,\cdots,i_d})
+ r^n_{i_1,i_2,\cdots,i_d}\nonumber\\
&:=& D_\tau^\alpha u^n_{i_1,i_2,\cdots,i_d} +r^n_{i_1,i_2,\cdots,i_d},
\end{eqnarray}
where $r^n$ represents the truncation error and
$
a_i=(i+1)^{1-\alpha} - i^{1-\alpha},\; i\geq 0.
$

Therefore, it holds that
\begin{eqnarray} \label{full-1}
D_\tau^\alpha  u^n_{i_1,i_2,\cdots,i_d} = &\sum_{j=1}^d \delta_{x_j}^2 u^n_{i_1,i_2,\cdots,i_d}+f(u^n_{i_1,i_2,\cdots,i_d})+r^n_{i_1,i_2,\cdots,i_d}+R^n_{i_1,i_2,\cdots,i_d}.
\end{eqnarray}

Let  $U^n_{i_1,i_2,\cdots,i_d}$ be numerical approximation to $u^n_{i_1,i_2,\cdots,i_d}$. Omitting the truncation errors in \eqref{full-1} and replacing $ u^n_{i_1,i_2,\cdots,i_d}$ by  $U^n_{i_1,i_2,\cdots,i_d}$,
we get the fully-discrete scheme
\begin{eqnarray} \label{full-2}
D_\tau^\alpha  U^n_{i_1,i_2,\cdots,i_d} = &\sum_{j=1}^d \delta_{x_j}^2 U^n_{i_1,i_2,\cdots,i_d}+f(U^n_{i_1,i_2,\cdots,i_d}).
\end{eqnarray}

\subsection{Convergence}

The convergence analysis depends heavily on the consistency error and the discrete fractional-type Gr\"onwall inequality.  For the purpose of the readability, we present the lemmas here and leave the proof in the next section.

 We now present the truncated error of L1 scheme for Caputo derivative.
\begin{Lem}\label{lem-trun}
Suppose that $u$ satisfies \eqref{assm-1}. Then, it holds that
\begin{align}\label{add-001}
|r_n|:=\left|D_{\tau}^{\alpha}u^n-\partial_{t_n}^{\alpha}u\right|\lesssim
\begin{cases}
\tau^{\sigma-\alpha} n^{-\min(1+\alpha,2-\sigma)},~~~~~&0<\sigma< 1,\\
\tau^{\sigma-\alpha} n^{-2+\sigma},&1<\sigma<2.
\end{cases}
\end{align}
\end{Lem}

The bound in \eqref{add-001} is sharp. It improves and generalizes the truncation error bound proved in \cite[Lemma 5.1]{sty17}.
Especially, when $\sigma=\alpha$, the bound is consistent to the estimate in \cite[Lemma 1]{gra18}.

We now consider a refined discrete fractional-type Gr\"onwall inequality, which will be used for convergence analysis (pointwise error esimate) of the L1 scheme studied in this paper.
\begin{Lem}[A refined discrete fractional Gr\"onwall inequality]\label{DGineq} Suppose $0<\alpha<1$ and $\tau>0$.
Let $y_i$, $0\leq i \leq N$, be a sequence of non-negative real numbers satisfying
\begin{align} \label{def-11}
D_\tau^\alpha y_n \leq \lambda y_n +\mu_1 n^{-\sigma_1} +\mu_2 n^{-\sigma_2}+\eta,\quad \quad \ \text{for}~ n=1,\dots,N,
\end{align}
where $\sigma_1>1$,$\sigma_2< 1$,$\lambda>0$, and $\mu_1,\mu_2\geq 0$.
Then there exists a constant $\tau_*=\sqrt[\alpha]{\frac{1}{2\Gamma(2-\alpha)\lambda}}$ such that $\tau<\tau^*$, it holds
\begin{align}\label{frac-re}
y_n \lesssim y_0+\mu_1\tau^\alpha n^{\alpha-1}+\mu_2\tau^\alpha n^{\alpha-\sigma_2}+\eta\tau^\alpha n^\alpha.
\end{align}
\end{Lem}
\begin{Rem}The case of $\mu_1=0$ and $\mu_2=0$ is investigated  in \cite[Theorem 2.6]{jin18} and \cite[Lemma 3.1]{li181} and \cite{liao-li}, respectively. The above Gr\"onwall inequality is a refined version, which is suitable to the pointwise error estimate for a class of scheme.
\end{Rem}

In this paper,  we mainly focus on the pointwise error estimate in time. Without loss of generality, we assume the nonlinear term $f:\mathbb{R}\rightarrow\mathbb{R}$
is Lipschitz continuous as
\begin{align}
|f(\xi_1)-f(\xi_2)|\leq L|\xi_1-\xi_2|,~~ \quad \textrm{for }\quad \xi_1,\xi_2\in\mathbb{R},
\label{LC}
\end{align}
where $L$ denotes the Lipschitz coefficient. In this assumption, we have following sharp pointwise error
 estimate
of scheme \eqref{full-2}.

\begin{Theo} \label{main}
Suppose that the system \eqref{i1}-\eqref{i1-ib} has a unique solution satisfying \eqref{assm-1}. Then, there exists a positive constant $\tau_0$,
such that when $\tau\leq \tau_0$,
the finite difference system \eqref{full-2}
admits a unique solution $U^n$, satisfying, for $n=1,2,\cdots,N$, it holds
\begin{eqnarray}
\|u^{n}-U^n\|_{\infty} \lesssim
\begin{cases}
\tau^{\sigma+1-\alpha} t_n^{\alpha-1}+ t_n^\alpha h^2,~~~~~~&0<\sigma<1,\\
\tau^{2-\alpha} t_n^{\alpha+\sigma-2}+ t_n^\alpha h^2,~~~~~~&1<\sigma\leq 2,
\end{cases}
\label{error-1}
\end{eqnarray}
where $u^n=[u_{1,1,\cdots,1}, u_{2,1,\cdots,1},\cdots, u_{M-1,1,\cdots,1}, u_{1,2,\cdots,1}, u_{2,2,\cdots,1},\cdots, u_{M-1,2,\cdots,1},\cdots, \\u_{1,1,\cdots,M-1},
 u_{2,1,\cdots,M-1},\cdots, u_{M-1,1,\cdots,M-1}]^T$
and $U^n=[U_{1,1,\cdots,1}, U_{2,1,\cdots,1},\cdots, \\
U_{M-1,1,\cdots,1}, U_{1,2,\cdots,1}, U_{2,2,\cdots,1},\cdots, U_{M-1,2,\cdots,1},\cdots, U_{1,1,\cdots,M-1}, U_{2,1,\cdots,M-1},\cdots,\\ U_{M-1,1,\cdots,M-1}]^T.$
\end{Theo}
\begin{proof} We firstly take the proof of one-dimensional problem for an example.
The exact solution of problem \eqref{i1} satisfies
\begin{align} \label{semi-ex}
D_\tau^\alpha u^n_i=\delta_x^2u_i^n+f(u_i^n)+r^n_i+R_i^n,
\end{align}
where the temporal truncation error $r^n_i = \partial_{t_n}^\alpha u - D_\tau^\alpha u^n_i$ and spatial truncation error $R_i^n= u_{xx}(x_i,t^n)-\delta_x^2u_i^n = \mathcal{O}( h^2).$

Suppose that $\|e^n\|_\infty=e^n_{i_0}$ and let  $e^n_i=u^n_i-U_n^i$. The error equation at the grid point $(x_{i_0},t_n)$ satisfies
\begin{align*}
D_\tau^\alpha e^n_{i_0}=\delta_x^2e_{i_0}^n+f(u^n_{i_0})-f(U^n_{i_0})+r^n_{i_0}+R_{i_0}^n,
\end{align*}
which can be rewritten as
\begin{eqnarray*}
\Big(\frac{\tau^{\!-\!\alpha}}{\Gamma(2\!-\!\alpha)} a_0\!+\!\frac{2}{h^2}\Big) e_{i_0}^n\!&=&\!\frac{e_{i_0\!-\!1}^n+e_{i_0+1}^n}{h^2}+\frac{\tau^{\!-\!\alpha}}{\Gamma(2\!-\!\alpha)}\sum_{j=1}^{n-1}(a_{n-j-1}-a_{n-j}) e^j_{i_0}\nonumber\\
&&+f(u^n_{i_0})-f(U^n_{i_0})+r^n_{i_0}+R_{i_0}^n.
\end{eqnarray*}

Therefore, we get
\begin{eqnarray} \label{semi-ex2}
&&\Big(\frac{\tau^{-\alpha}}{\Gamma(2-\alpha)} a_0+\frac{2}{h^2}\Big) |e_{i_0}^n|\nonumber\\
&=&\Big|\frac{e_{i_0-1}^n+e_{i_0+1}^n}{h^2}+\frac{\tau^{-\alpha}}{\Gamma(2-\alpha)}\sum_{j=1}^{n-1}(a_{n-j-1}-a_{n-j}) e^j_{i_0}+f(u^n_{i_0})-f(U^n_{i_0})+r^n_{i_0}+R_{i_0}^n\Big|\nonumber\\
&=&\Big|\frac{e_{i_0-1}^n+e_{i_0+1}^n}{h^2}+\frac{\tau^{-\alpha}}{\Gamma(2-\alpha)}\sum_{j=1}^{n-1}(a_{n-j-1}-a_{n-j}) e^j_{i_0}+f(u^n_{i_0})-f(U^n_{i_0})+r^n_{i_0}+R_{i_0}^n\Big|\nonumber\\
&\leq &\Big|\frac{2e_{i_0}^n}{h^2}\Big|+\frac{\tau^{-\alpha}}{\Gamma(2-\alpha)}\sum_{j=1}^{n-1}(a_{n-j-1}-a_{n-j}) |e^j_{i_0}|+L|e_{i_0}^n|+|r^n_{i_0}|+|R_{i_0}^n|,
\end{eqnarray}
where we use the facts $a_i>a_{i+1}$.

Equation \eqref{semi-ex2} further implies that
\begin{eqnarray*}
\Big(\frac{\tau^{-\alpha}}{\Gamma(2-\alpha)} a_0\Big) |e_{i_0}^n|&\leq &\frac{\tau^{-\alpha}}{\Gamma(2-\alpha)}\sum_{j=1}^{n-1}(a_{n-j-1}-a_{n-j}) |e^j_{i_0}|+L|e_{i_0}^n|+|r^n_{i_0}|+|R_{i_0}^n|,
\end{eqnarray*}
which can be rewritten as
\begin{eqnarray}\label{semi-ex3}
D_\tau^\alpha |e_{i_{0}}^n|\leq L|e_{i_0}^n|+|r^n_{i_0}|+|R_{i_0}^n|.
\end{eqnarray}

For $0\leq \sigma<1$,  it holds that $|r^n_{i_0}|\lesssim\tau^{\sigma-\alpha} n^{-\min(1+\alpha,2-\alpha) }$ with $\min(1+\alpha,2-\alpha)>1$ (by lemma \ref{lemmar}). Together with \eqref{semi-ex3} and Lemma \ref{DGineq}, we have
\[ \|e^n\|_\infty = |e_{i_{0}}^n| \lesssim    \tau^\sigma n^{\alpha-1} + h^2  = \tau^{\sigma+1-\alpha} t_n^{\alpha-1}+t_n^\alpha h^2.   \]

For $1< \sigma\leq 2$,  it holds that $|r^n_{i_0}|\lesssim\tau^{\sigma-\alpha} n^{\sigma-2 }$ with $2-\sigma<1$ (by lemma \ref{lemmar}). Together with \eqref{semi-ex3} and Theorem \ref{DGineq}, we have
\[ \|e^n\|_\infty = |e_{i_{0}}^n| \lesssim    \tau^\sigma n^{\alpha+\sigma-2} + h^2  = \tau^{2-\alpha} t_n^{\alpha+\sigma-2}+  t_n^\alpha h^2.   \]
 The extension to the multi-dimensional problems can be obtained similarly. This completes the proof of the main results.
\end{proof}

\begin{Rem} ~The convergence results imply that,  when $t_n\rightarrow 0$, it holds that
\begin{eqnarray} \label{001}
\max_{1\leq n \leq N}\|u^{n}-U^n\|_{\infty} \lesssim \tau^\sigma +\tau^\alpha h^2.
\end{eqnarray}
When $t$ is far away form $0$, it holds that
\begin{eqnarray}
\|u^{n}-U^n\|_{\infty} \lesssim
\begin{cases}
\tau^{\sigma+1-\alpha}+ h^2,~~~~~~&\sigma \in (0,1),\\
\tau^{2-\alpha} +  h^2,~~~~~~&\sigma \in (1,2].
\end{cases}
\label{error-3}
\end{eqnarray}
Meanwhile,  if we test the maximum error in the whole domain $\Omega \times[0,T]$, we have
\begin{eqnarray}
\max_{1\leq n \leq N}\|u^{n}-U^n\|_{\infty} \lesssim
\begin{cases}
\tau^{\sigma}+ h^2,~~~~~~&\sigma \in (0,1)\cup (1,2-\alpha),\\
\tau^{2-\alpha} +  h^2,~~~~~~&\sigma \in [2-\alpha,2].
\end{cases}
\label{error-2}
\end{eqnarray}
\end{Rem}

\begin{Rem} Scheme \eqref{full-2} is fully implicit. If the nonlinear term is approximated by using the Newton linearized method, we have the following fully-discrete scheme, i.e.,
\begin{eqnarray} \label{full-add1}
D_\tau^\alpha  U^n_{i_1,i_2,\cdots,i_d} &= &\sum_{j=1}^d \delta_{x_j}^2 U^n_{i_1,i_2,\cdots,i_d}+f_1(U^{n-1}_{i_1,i_2,\cdots,i_d})\nonumber\\
&&+f(U^{n-1}_{i_1,i_2,\cdots,i_d})(U^{n}_{i_1,i_2,\cdots,i_d}-U^{n-1}_{i_1,i_2,\cdots,i_d}),
\end{eqnarray}
where $f_1(U^{n-1}_{i_1,i_2,\cdots,i_d})=\frac{\partial}{\partial u}f|_{u=U^{n-1}_{i_1,i_2,\cdots,i_d}}$. We can have the same error estimates as \eqref{error-1}. The proof is similar
to that done above. Meanwhile, the proof of the convergence results is based on  Lipschitz condition. If $f\in C^1(\mathbb{R})$, the main results still hold. This is because
\begin{eqnarray}
\| U^{n}\|_{\infty} \leq \| u^{n}\|_{\infty} + \|u^n- U^{n}\|_{\infty}\lesssim  \| u^{n}\|_{\infty} + 1,
\label{inverse}
\end{eqnarray}whenever the stepsizes are sufficiently small. Due to the boundedness of $\|U^{n}\|_{\infty}$,
 we have
 \[
 |f(u^{n}_{i_1,i_2,\cdots,i_d})-f(U^{n}_{i_1,i_2,\cdots,i_d})|
\lesssim |u^{n}_{i_1,i_2,\cdots,i_d}-U^{n}_{i_1,i_2,\cdots,i_d}|.
\]
Then, the results  can be proved by using similar method.
\end{Rem}

\begin{Rem}The assumption $\sigma \in (0,\alpha]$ is reasonable and widely accepted. Here we refer readers to \cite{mas19,jin19} for the detailed theoretical results. Suppose that the solution is smoother, i.e.,
$\sigma>\alpha$, some additional hypothesis should be added (see \cite{sty17}). However, it is still quite restrictive due to the strong hypotheisis. In this work, we assume
that $\sigma\in(0,1)\cup(1,2)$ in order to make the current analysis extendable.
\end{Rem}

\begin{Rem} In the previous results, most time-stepping error estimates of L1 scheme focus on the linear subdiffusion equations with reaction terms in the type of $f(u) = -u$ and under the assumption $\sigma=\alpha$ \cite{sty17,gra18,jin19} . For the reaction terms in the type of $f(u) = u$ or nonlinear problems, the proof is essentially different, which requires a refined discrete fractional-type Gr\"onwall inequality presented in Lemma \ref{DGineq} above.
\end{Rem}

\begin{Rem}
 There are some error estimates of L1 scheme for nonlinear problems under the assumption $\sigma=2$, e.g., \cite{li183,li181}. In such cases, the truncation error is independent of $n$. Then,  the convergence order is  $2-\alpha$ and unchanged. Such results can be concluded from the present results, but their proof
is totally different.
\end{Rem}

\begin{Rem} In \cite{liao-li}, another version of the discrete fractional-type Gr\"onwall inequality is developed. Based on the inequality and the regularity assumption \eqref{assm-1}, one has the optimal error estimate $\mathcal{O}(\tau^\sigma)$ for uniform time mesh, and $\mathcal{O}(\tau^{\min\{\gamma\sigma,2-\alpha\}})$ for graded mesh, where $\gamma$ represents the parameter reflecting the graded mesh $t_k = T(k/N)^\gamma$. Note that in \cite{liao-li} there is a maximum in the discrete fractional-type Gr\"onwall inequality and the kernels are hard to explicitly be expressed, it fails to obtain the pointwise error estimate.
\end{Rem}

\section{Proof of the lemmas}\label{sec3}
In this section, we present a detailed proof of the lemmas in the previous section.

\subsection{ Proof of Lemma \ref{lem-trun} }
\begin{proof}
We now estimate the truncation errors \eqref{add-001} in two cases.

{\bf Case A:} n=1.
It holds that
\begin{eqnarray}\label{tes-0}
&&\Big|D_{\tau}^{\alpha}u^1-\partial_{t_1}^{\alpha}u\Big|\nonumber\\
&=&\Big|\frac{1}{\Gamma(1-\alpha)}\int_{0}^{t_{1}}(t_1-s)^{-\alpha}[u(x,\xi_1)-\frac{\partial}{\partial s}u(x,s)]ds\Big|~~~\quad (\xi_1\in(0,t_1))\nonumber\\
&\leq & \left|\frac{1}{\Gamma(1-\alpha)}\int_{0}^{t_{1}}(t_1-s)^{-\alpha}\frac{\partial}{\partial \xi_1}u(x,\xi_1) ds\right|+\left|\frac{1}{\Gamma(1-\alpha)}\int_{0}^{t_{1}}(t_1-s)^{-\alpha}\frac{\partial}{\partial s}u(x,s)ds\right|\nonumber\\
&\lesssim & \left|\int_{0}^{t_{1}}(t_1-s)^{-\alpha}s^{\sigma-1} ds\right|\nonumber\\
&=& \left| t_1^{\sigma-\alpha}\beta(1-\alpha,\sigma) \right|\nonumber\\
&\lesssim & \tau^{\sigma-\alpha}.
\end{eqnarray}

{\bf Case B: $n>1$.} It follows from the definitions of the discrete and continuous fractional operator that
\begin{align}\label{tes-1}
D_{\tau}^{\alpha}u^n-\partial_{t_n}^{\alpha}u&=\sum\limits_{k=0}^{n\!-\!1}\frac{1}{\Gamma(1\!-\!\alpha)}\int_{t_k}^{t_{k\!+\!1}}(t_n\!-\!s)^{\!-\!\alpha}[\frac{u^{k\!+\!1}\!-\!u^k}{\tau}-\frac{\partial}{\partial s}u(x,s)]ds\nonumber\\
&=
\sum\limits_{k=0}^{n\!-\!1}\frac{\!-\!\alpha}{\Gamma(1\!-\!\alpha)}\int_{t_k}^{t_{k+1}}(t_n\!-\!s)^{\!-\!\alpha\!-\!1}[\frac{u^{k+1}-u^k}{\tau}(s-t_k)-(u(x,s)-u^k)]ds\nonumber\\
:&=\sum\limits_{k=0}^{n-1}R_{nk}.
\end{align}

For $1\leq k \leq n-1$, by standard interpolation theory, it holds that
\begin{align*}
\left|R_{nk}\right|&\lesssim \tau^2(\max\limits_{s\in[t_k,t_{k+1}]}\left|u_{tt}(x,s)\right|\int_{t_k}^{t_{k+1}}(t_n-s)^{-\alpha-1}ds)\nonumber\\
&\lesssim\tau^3t_k^{\sigma-2}(t_n-t_{k+1})^{-\alpha-1}\nonumber\\
&=\tau^{\sigma-\alpha}k^{\sigma-2}(n-k+1)^{-\alpha-1}.
\end{align*}
As a result, we get
\begin{align}\label{tes-2}
\sum\limits_{k=1}^{\lceil n/2\rceil-1}\left|R_{nk}\right|
&\lesssim
\tau^{\sigma-\alpha}n^{-(1+\alpha)}\sum\limits_{k=1}^{\lceil\frac{n}{2}\rceil-1}k^{\sigma-2}
\lesssim
\begin{cases}
\tau^{\sigma-\alpha}n^{-(1+\alpha)},~~~~~~~&0<\sigma<1,\\
\tau^{\sigma-\alpha}n^{-2-\alpha+\sigma},~~~~&1<\sigma\leq 2.
\end{cases}
\end{align}
For $\lceil n/2\rceil\leq k \leq n-1$, it holds that
\begin{align}\label{tes-3}
\left|R_{nk}\right|&\lesssim\tau^{2}t_{k}^{\sigma-2}\int_{t_k}^{t_{k+1}}(t_n-s)^{-\alpha-1}ds\nonumber\\
&= \tau^{\sigma}k^{\sigma-2}\int_{t_k}^{t_{k+1}}(t_n-s)^{-\alpha-1}ds\nonumber\\
&\lesssim \tau^{\sigma}n^{\sigma-2}\int_{t_k}^{t_{k+1}}(t_n-s)^{-\alpha-1}ds.
\end{align}
Therefore,
\begin{align}
\sum\limits_{k=\lceil n/2\rceil}^{n-2}\left|R_{nk}\right|\lesssim \tau^{\sigma}n^{\sigma-2}\int_{t_{\ulcorner n/2\rceil}}^{t_{n-2}}(t_n-s)^{-\alpha-1}ds \lesssim \tau^{\sigma-\alpha}n^{\sigma-2}.
\end{align}
For $k=0$, noting that
\begin{align*}
\left|\frac{u^1-u^0}{\tau}(s-0)-(u(x,s)-u^0)\right|\leq 2\int_0^{t_1}\left|u_t(x,s)\right|ds\leq 2\tau^{\sigma},
\end{align*}
we get
\begin{align}\label{tes-4}
\left|R_{n0}\right| &=\left|\frac{-\alpha}{\Gamma(1-\alpha)}\int_{0}^{t_1}(t_n-s)^{-\alpha-1}[\frac{u^1-u^0}{\tau}(s-0)-(u(x,s)-u^0)]ds\right|\nonumber\\
&\lesssim\tau^{\sigma}\int_0^{t_1}(t_n-s)^{-\alpha-1}ds\nonumber\\
&=\tau^{\sigma}\tau(t_n-t_1)^{-\alpha-1}\nonumber\\
&=\tau^{\sigma-\alpha}n^{-\alpha-1}.
\end{align}
For $k=n-1$, it holds that
\begin{align}\label{tes-5}
\left|T_{n,n-1}\right|& =\frac{1}{\Gamma(1-\alpha)}\int_{t_{n-1}}^{t_{n}}(t_n-s)^{-\alpha}[\frac{u^{n}-u^{n-1}}{\tau}-\frac{\partial}{\partial s}u(x,s)]ds\nonumber\\
&\leq \frac{1}{\Gamma(1-\alpha)} \Big| \frac{\partial}{\partial s}u(x,\xi_{n-1}^*)-\frac{\partial}{\partial s}u(x,\xi_{n-1}^{**})\Big|\int_{t_{n-1}}^{t_{n}}(t_n-s)^{-\alpha}ds\nonumber\\
&\lesssim \tau\Big| \frac{\partial^2}{\partial s^2}u(x,\xi_{n-1}^{***})\Big|\int_{t_{n-1}}^{t_{n}}(t_n-s)^{-\alpha}ds\nonumber\\
&\leq \tau^{2-\alpha}t_{n-1}^{\sigma-2}\nonumber\\
&=\tau^{\sigma-\alpha}n^{\sigma-2},
\end{align}
where $\xi_{n-1}^*, \xi_{n-1}^{**}$ and $\xi_{n-1}^{***}$ are constants between $t_{n-1}$ and $t_n$.

Now, together with the estimates \eqref{tes-1}-\eqref{tes-5}, we obtain the finial results.
\end{proof}

\subsection{Discrete complementary convolutions}

In order to prove the refined discrete fractional-type Gr\"onwall inequality, we first introduce the useful properties of the discrete complementary convolutions. They are given in \cite{liao-li} can be reduced to the following lemma by removing the factor with respect to $\tau$.
\begin{Lem}[\cite{gra18}]\label{lemma:recursionCoefficient}
Let $\{ p_n \}$ be a sequence defined by
\begin{align}
p_{0}=1,\quad p_{n}=\sum_{j=1}^{n}(a_{j-1}-a_j)p_{n-j},\quad n\geq1.
\label{p}
\end{align}
Then it holds that
\begin{itemize}
\item [{(i)}]\quad $ p_n <(n+1)^{\alpha-1}$,
\quad $1\leq k\leq n$,
\label{Lem-0}
\item [{(ii)}]\quad  $\sum_{j=k}^{n}p_{n-j}a_{j-k}=1$,
\quad $1\leq k\leq n$,
\label{Lem-1}
\item [{(iii)}] \quad $\Gamma(2-\alpha)\sum_{j=1}^{n}p_{n-j}\leq\frac{n^{\alpha}}{\Gamma(1+\alpha)}$.
\label{Lem-2}
\end{itemize}
\end{Lem}

\begin{Lem}\label{lemmar}
Let $\{ p_n \}$ be a sequence defined in \eqref{p}. Then, it holds that
\begin{align}
  \frac{1}{\Gamma(1-\gamma)}\sum_{j=1}^n p_{n-j}j^{-\gamma} \lesssim
\begin{cases}
 \frac{n^{\alpha-1}}{\Gamma(1-\gamma)},~~~~~&\gamma>1,\\
 \frac{n^{\alpha-\gamma}}{\Gamma(1-\gamma+\alpha)},&\gamma<1.
\end{cases}
\end{align}
\end{Lem}
\begin{proof}
It follows from \eqref{Lem-0} in Lemma \ref{lemma:recursionCoefficient} and Lemma \ref{lem-trun} that
\begin{eqnarray*}
\sum_{j=1}^n p_{n-j}j^{-\gamma}  &\lesssim&  \sum_{j=1}^{n-1}(n-j+1)^{\alpha-1}j^{-\gamma }+n^{-\gamma}\nonumber\\
 &\lesssim&  (\frac{n}{2})^{\alpha-1}\sum_{j=1}^{\lceil n/2\rceil}j^{-\gamma }+(\frac{n}{2})^{\alpha-\gamma}\sum_{j={\lceil n/2\rceil}+1}^{n-1}(n-j+1)^{\alpha-1}j^{-\gamma}+\tau^{\sigma-\alpha}n^{-\gamma}\nonumber\\
  &\lesssim&  \Big( n^{\alpha-1}+n^{\alpha-\gamma}\int_0^{n}s^{-\alpha}(n-s)^{\alpha-1}ds+ n^{-\gamma}  \Big)\nonumber\\
   &\lesssim& n^{\alpha-1},
\end{eqnarray*}
where we have noted that
\[ \sum_{j=1}^nj^{-\gamma}<\infty,\quad \gamma>1. \]
For $\gamma<1$, it holds that
\begin{eqnarray}
 \frac{1}{\Gamma(1-\gamma)}\sum_{j=1}^n p_{n-j} j^{-\gamma} &\lesssim&  \frac{1}{\Gamma(1-\gamma)} \sum_{j=1}^{n-1}(n-j+1)^{\alpha-1}j^{-\gamma}+ \frac{1}{\Gamma(1-\gamma)}n^{-\gamma}\nonumber\\
                       &\lesssim&   \frac{1}{\Gamma(1-\gamma)}\sum_{j=1}^{n-1}\int_{j-1}^j s^{-\gamma}(n+1-s)^{\alpha-1}+ \frac{1}{\Gamma(1-\gamma)}n^{-\gamma}\nonumber\\
                        &\lesssim&  \frac{1}{\Gamma(1-\gamma)}\int_{0}^{n} s^{-\gamma}(n-s)^{\alpha-1}+ \frac{1}{\Gamma(1-\gamma)}n^{-\gamma}\nonumber\\
                        &\lesssim&  \frac{n^{\alpha-\gamma}}{\Gamma(1-\gamma)}\beta{(1-\gamma, \alpha)}+ \frac{1}{\Gamma(1-\gamma)}n^{-\gamma}\nonumber\\
                        &\lesssim&  \frac{n^{\alpha-\gamma}}{\Gamma(1-\gamma+\alpha)}.
\end{eqnarray}
This completes the proof.
\end{proof}

\begin{Lem}\label{lemma:2}
Let $Z_1=(n^{-\sigma_1}, (n-1)^{-\sigma_1},\cdots,1^{-\sigma_1})^T,  ~Z_2=(n^{-\sigma_2}, (n-1)^{-\sigma_2},\cdots,1^{-\sigma_2})^T, Z_3=(1, 1,\cdots,1)^T \in R^n$ with $\sigma_1>1$ and $\sigma_2<1$,  and
\begin{equation}\label{matr-0}
 J=2\Gamma(2-\alpha)\lambda\tau^\alpha\left[\begin{matrix}
    0 &~ p_1   &~\cdots &p_{n-2}  & p_{n-1} \\
    0&~  0    &~\cdots& ~p_{n-3}& ~p_{n-2}\\
    \vdots&~ \vdots    &~ \ddots & ~\vdots& ~\vdots\\
    0 & ~0  &~\cdots&~ 0&~p_1 \\
    0& ~0& ~\cdots& ~0& ~0\\
\end{matrix}\right]_{n\times n}.
\end{equation}
Then, it holds that
\begin{itemize}
\item [{(i)}] $J^i=0,~~i\geq n$;
\item [{(ii)}]
$J^m Z_2 \leq \frac{\Gamma(1-\sigma_2)(2\Gamma(2-\alpha)\lambda \tau^\alpha)^m}{\Gamma(1-\sigma_2+m\alpha)}
\Big( n^{m\alpha-\sigma_2},  (n-1)^{m\alpha-\sigma_2},
\cdots,1^{m\alpha-\sigma_2}\Big)^T$,\\
~~ $m=0,1,2,\cdots$;
\item [{(iii)}]
$\sum\limits_{j=1}^iJ^jZ_2 \lesssim \tau^\alpha\Big( n^{\alpha-\sigma_2}E_{\alpha,1-\sigma_2}(2\Gamma(2-\alpha)\lambda t_n^\alpha), (n-1)^{\alpha-\sigma_2} E_{\alpha,1-\sigma_2}(2\Gamma(2-\alpha)\lambda t_{n-1}^\alpha) ,
\cdots,\nonumber\\
 1^{\alpha-\sigma_2}E_{\alpha,1-\sigma_2}(2\Gamma(2-\alpha)\lambda t_1^\alpha))\Big)^T$,~~ $i \geq n$,
\item [{(iv)}] $J Z_1 \lesssim 2\Gamma(2-\alpha)\lambda\tau^\alpha
\Big( n^{\alpha-1}, (n-1)^{\alpha-1},
\cdots,1^{\alpha-1}\Big)^T$;
\item [{(v)}]
$\sum\limits_{j=1}^\infty J^jZ_1\lesssim 2\Gamma(2-\alpha)\lambda\tau^\alpha\Big( n^{\alpha-1}E_{\alpha,\alpha}(2\Gamma(2-\alpha)\lambda t_n^\alpha), (n-1)^{\alpha-1} E_{\alpha,\alpha}(2\Gamma(2-\alpha)\lambda t_{n-1}^\alpha) ,
\cdots,
 1^{\alpha-1}E_{\alpha,\alpha}(2\Gamma(2-\alpha)\lambda t_1^\alpha))\Big)^T$
,~~ $i \geq n$;
\item [{(vi)}]
$\sum\limits_{j=1}^\infty J^jZ_3\lesssim \tau^\alpha\Big( n^{\alpha}E_{\alpha,1-\sigma_2}(2\Gamma(2-\alpha)\lambda t_n^\alpha), (n-1)^{\alpha} E_{\alpha,1-\sigma_2}(2\Gamma(2-\alpha)\lambda t_{n-1}^\alpha) ,
\cdots,\nonumber\\
 1^{\alpha}E_{\alpha,1-\sigma_2}(2\Gamma(2-\alpha)\lambda t_1^\alpha))\Big)^T$.
\end{itemize}
\end{Lem}
\begin{proof}
Since that $J$ is an upper triangular matrix, one can check (i) holds .

Now, we prove (ii) by using the mathematical induction.
Firstly, (ii) holds for $m=0$. Suppose that (ii) holds for $m=k$.
\begin{eqnarray}
J^{k+1}Z_2&=& J (J^{k}Z_2) \leq  \frac{\Gamma(1-\sigma_2)(2\Gamma(2-\alpha)\lambda \tau^\alpha)^k}{\Gamma(1-\sigma_2+k\alpha)}J\Big( n^{k\alpha-\sigma_2},  (n-1)^{k\alpha-\sigma_2},
\cdots,1^{k\alpha-\sigma_2}\Big)^T\nonumber\\
&= &
\frac{\Gamma(1-\sigma_2)(2\Gamma(2-\alpha)\lambda \tau^\alpha)^{k+1}}{\Gamma(1-\sigma_2+k\alpha)}\Big( \sum_{i=1}^{n-1}p_{n-i}i^{1-\sigma_2+k\alpha},  \sum_{i=1}^{n-2}p_{n-1-i}i^{1-\sigma_2+k\alpha},
\cdots,0\Big)^T\nonumber\\
&\leq &
\frac{\Gamma(1-\sigma_2)(2\Gamma(2-\alpha)\lambda \tau^\alpha)^{k+1}}{\Gamma(1-\sigma_2+k\alpha+\alpha)}\Big( n^{k\alpha-\sigma_2+\alpha},  (n-1)^{k\alpha-\sigma_2+\alpha},
\cdots,1^{k\alpha-\sigma_2+\alpha}\Big)^T,\nonumber\\
\end{eqnarray}
where in the last inequality, we have noted Lemma \ref{lemmar}.
Therefore, (ii) holds for $m=k+1$. This completes the mathematical induction and the proof.

Next, we show (iii) holds. Note that (i) implies that
$\sum_{j=1}^iJ^jZ_2
 = \sum_{j=1}^{n-1}J^{j}Z_2$ for $i\geq n$, and by (iii), we get
\begin{eqnarray*}
\sum_{j=1}^{n-1}J^{j}Z_2
&\lesssim&  \tau^\alpha\sum_{j=1}^{n-1}
\Big( \frac{(2\Gamma(2-\alpha)\lambda \tau^\alpha n^\alpha)^j}{\Gamma(1-\sigma_2+j\alpha)}n^{\alpha-\sigma_2}, \frac{(2\Gamma(2-\alpha)\lambda \tau^\alpha(n-1)^\alpha)^j}{\Gamma(1-\sigma_2+j\alpha)}(n-1)^{\alpha-\sigma_2},\nonumber\\
&&\cdots,\frac{(2\Gamma(2-\alpha)\lambda \tau^\alpha1^\alpha)^j}{\Gamma(1-\sigma_2+j\alpha)}1^{\alpha-\sigma_2}\Big)^T \nonumber\\
&\lesssim&  \tau^\alpha\Big( n^{\alpha-\sigma_2}E_{\alpha,1-\sigma_2}(2\Gamma(2-\alpha)\lambda t_n^\alpha), (n-1)^{\alpha-\sigma_2} E_{\alpha,1-\sigma_2}(2\Gamma(2-\alpha)\lambda t_{n-1}^\alpha) ,
\cdots,\nonumber\\
&& 1^{\alpha-\sigma_2}E_{\alpha,1-\sigma_2}(2\Gamma(2-\alpha)\lambda t_1^\alpha))\Big)^T
\end{eqnarray*}
where
\[ E_{\alpha,\beta}(z)=\sum_{k=0}^\infty \frac{z^k}{\Gamma(k\alpha + \beta)}.  \]

For inequality (iv), one can check it holds by using Lemma \ref{lemmar}.

For inequality (v), we have
\[ \sum\limits_{j=1}^\infty J^jZ_1= \sum\limits_{j=0}^\infty J^j (JZ_1)\lesssim  2\Gamma(2-\alpha)\lambda\tau^\alpha\sum\limits_{j=0}^\infty J^j
\Big( n^{\alpha-1}, (n-1)^{\alpha-1},
\cdots,1^{\alpha-1}\Big)^T.\]
Then, the result (v) holds by using  (iii) and the fact $1-\alpha<1$.

Finally, let $\sigma_2=0$ in (iii), one can obtain (vi). This completes the proof.
\end{proof}

\subsection{Proof of the Gr\"onwall inequality} Now, we are ready to prove the refined Gr\"onwall inequality in Lemma \ref{DGineq}.
\begin{proof}
It follows from the definition of $L1$-approximation \eqref{l1sch} that
\begin{equation} \label{GWN}
\sum_{k=1}^ja_{j-k}{\delta_t y_k}\leq \Gamma(2-\alpha)\tau^{\alpha}
\lambda_1 y_j
+\Gamma(2-\alpha)\tau^{\alpha}(\mu_1 j^{-\sigma_1} +\mu_1 j^{-\sigma_2}+\eta).
\end{equation}
Multiplying \eqref{GWN} by $p_{n-j}$ and then summing over
for $j$ from $1$ to $n$, we get
\begin{align}\label{ineq1.1}
\sum_{j=1}^np_{n\!-\!j}\sum_{k=1}^ja_{j\!-\!k}{\delta_t y_k}\!
\leq \! \Gamma(2\!-\!\alpha)\tau^{\alpha}\sum_{j=1}^np_{n\!-\!j}\lambda_1y_j
+\Gamma(2\!-\!\alpha)\tau^{\alpha}\sum_{j=1}^{n}p_{n\!-\!j}(j^{-\sigma_1}+j^{-\sigma_2}+\eta).
\end{align}
Applying \eqref{Lem-1} in Lemma \ref{lemma:recursionCoefficient} gives
\begin{align}\label{ineq1.2}
\sum_{j=1}^np_{n-j}\sum_{k=1}^ja_{j-k}{\delta_ty_k}=\sum_{k=1}^n\delta_ty_k\sum_{j=k}^np_{n-j}a_{j-k}
=\sum_{k=1}^n\delta_ty_k=y_n-y_0, \quad n\geq1.
\end{align}

Substituting \eqref{ineq1.2} into \eqref{ineq1.1} gives
\begin{align}
y_n
&\leq y_0+ \Gamma(2\!-\!\alpha)\tau^{\alpha}\sum_{j=1}^np_{n\!-\!j}\lambda_1y_j
+\Gamma(2\!-\!\alpha)\tau^{\alpha}\sum_{j=1}^{n}p_{n\!-\!j}(j^{-\sigma_1}+j^{-\sigma_2}+\eta)\nonumber\\
&= y_0+\Gamma(2-\alpha)\tau^{\alpha}y_n+ \Gamma(2-\alpha)\tau^{\alpha}\lambda_1\sum_{j=1}^{n-1}p_{n-j}y_j\nonumber\\
&~~~+\Gamma(2\!-\!\alpha)\tau^{\alpha}\sum_{j=1}^{n}p_{n\!-\!j}(j^{-\sigma_1}+j^{-\sigma_2}+\eta)\nonumber\\
&\leq y_0+\frac{1}{2}y_n+ \Gamma(2-\alpha)\tau^{\alpha}\lambda_1\sum_{j=1}^{n-1}p_{n-j}y_j
+\Gamma(2\!-\!\alpha)\tau^{\alpha}\sum_{j=1}^{n}p_{n\!-\!j}(j^{-\sigma_1}+j^{-\sigma_2}+\eta),\label{ineq1.4}
\end{align}
whenever  $\tau\leq \sqrt[\alpha]{\frac{1}{2\Gamma(2-\alpha)\lambda_1}}$.

Therefore, we get
\begin{eqnarray}\label{ineq1.5}
y_n\!
\leq\! 2 y_0\!+\!2\Gamma(2\!-\!\alpha)\tau^{\alpha}\lambda_1\sum_{j=1}^{n-1}p_{n-j}y_j
\!+\!\Gamma(2\!-\!\alpha)\tau^{\alpha}\sum_{j=1}^{n}p_{n\!-\!j}(j^{-\sigma_1}\!+\!j^{-\sigma_2}\!+\!\eta).
\end{eqnarray}

Let $Y=(y_n,y_{n-1},\cdots,y_1)^T$ and $Z=C\Gamma(2\!-\!\alpha)\tau^{\alpha}(\mu_1Z_1+\mu_2Z_2+\eta Z_3)$.
Then,  \eqref{ineq1.5} can be rewritten in a matrix form as follows
\begin{equation}\label{matr-1}
Y \leq JY + JZ.
\end{equation}

As a result, we have
\begin{eqnarray}
\label{gron-ineq}
 Y &\leq& J V + JZ
\leq J(JY+JZ)+ JZ
=J^2Y
+\sum_{j=1}^2J^jZ
\nonumber\\
 &\leq&\cdots \leq J^nY+\sum_{j=1}^{n}J^jZ =  \mu_1\sum_{j=1}^{n-1}J^jZ_1+ \mu_2\sum_{j=1}^{n-1}J^jZ_2 +\eta\sum_{j=1}^{n-1}J^jZ_3\;.
 \end{eqnarray}
Now, by (iv),(v) and (vi) in Lemma \ref{lemma:2}, we obtain \eqref{frac-re} and complete the proof. \end{proof}

\section{Numerical examples}\label{sec4}

In this section, we present several numerical results to illustrate the sharp pointwise error estimates.\\
{\bf Example 1.} ~Consider the one-dimensional nonlinear subdiffusion problems
  \begin{align}
~\partial_t^\alpha u= u_{xx}+\sqrt{1+u^2}+g(x,t),
\quad
(x,t)\in (0,\pi)\times(0,t_N],
\label{exam-1d1}
\end{align}
where the initial condition and $g(x,t)$ are specially chosen such that the problem admits an exact solution in the form of
\[ u(x,t)= t^\sigma \sin( x).\]

\begin{table}[ht]
\begin{center}
\caption{Maximum errors at $t=1$ and convergence orders in temporal direction with $M=1000$ for Example 1}
\begin{tabular}{cc|cccccc|cc}
     \hline
  $\alpha$    &$\sigma$        $ \backslash N$& $10$       &$20$           &$40$      &$80$       &$160$     &$Rate$ & expected order \\[1ex]
       \hline
     0.4   & $0.1$                   & 3.50e-2    & 2.05e-2        & 1.22e-2  & 7.37e-3   &4.46e-3   & 0.73 & $\sigma+1-\alpha$\\[1ex]
           & $0.4$                   & 1.12e-2    & 5.21e-3        & 2.46e-3  & 1.18e-3   &5.71e-4   & 1.07 & $\sigma+1-\alpha$\\[1ex]
           & $0.6$                   & 5.13e-3    & 2.10e-3        & 8.70e-4  & 3.63e-4   &1.53e-4   & 1.25 & $\sigma+1-\alpha$\\[1ex]
           & $1.2$                   & 3.98e-3    & 1.57e-3        & 6.17e-4  & 2.41e-4   &9.40e-5   & 1.36 & $2-\alpha$\\[1ex]
           & $1.8$                   & 1.30e-2    & 5.03e-3        & 1.93e-3  & 7.38e-4   &2.81e-4   & 1.39 & $2-\alpha$\\[1ex]
     \hline
   0.6     & $0.4$                   & 3.27e-2    & 1.78e-2        & 9.89e-3  & 5.54e-3   &3.13e-3   & 0.84 & $\sigma+1-\alpha$\\[1ex]
           & $0.6$                   & 1.51e-2    & 7.28e-3        & 3.54e-3  & 1.73e-3   &8.51e-4   & 1.03 & $\sigma+1-\alpha$\\[1ex]
           &$0.8$                   & 5.68e-3    & 2.49e-3        & 1.09e-3  & 4.76e-4   &2.08e-4    & 1.19 & $\sigma+1-\alpha$\\[1ex]
           & $1.2$                   & 3.98e-3    & 1.57e-3        & 6.17e-4  & 2.41e-4   &9.40e-5   & 1.36 & $2-\alpha$\\[1ex]
           & $1.8$                   & 1.30e-2    & 5.03e-3        & 1.93e-3  & 7.38e-4   &2.81e-4   & 1.39 & $2-\alpha$\\[1ex]
  \hline
\label{1d11}
\end{tabular}
\end{center}
\end{table}

\begin{table}[ht]
\begin{center}
\caption{Maximum errors as $t\rightarrow 0$ and convergence orders in temporal direction with $M=1000$ and $N=10$ for Example 1}
\begin{tabular}{cc|cccccc|cc}
     \hline
  $\alpha$    &$\sigma$    $\backslash t_N$& $1e-3$       &$1e-4$           &$1e-5$      &$1e-6$       &$1e-7$     &$Rate$ & expected order \\[1ex]
       \hline
     0.4   & $0.1$                   & 2.90e-2    & 2.39e-2        & 1.92e-2  & 1.54e-2   &1.23e-2   & 0.09 & $\sigma$\\[1ex]
           & $0.4$                   & 1.20e-3    & 5.02e-4        & 2.05e-4  & 8.22e-5   &3.28e-5   & 0.40 & $\sigma$\\[1ex]
           & $0.6$                   & 1.38e-4    & 3.67e-5        & 9.43e-6  & 2.39e-6   &6.03e-7   & 0.59 & $\sigma$\\[1ex]
           & $1.2$                   & 5.61e-7    & 3.70e-8        & 2.38e-9  & 1.51e-10  &9.57e-12   & 1.19 & $\sigma$\\[1ex]
           & $1.8$                   & 2.87e-8    & 4.71e-10       &7.58e-12  & 1.21e-13   &1.92e-15   & 1.80 & $\sigma$\\[1ex]
     \hline
   0.6     & $0.4$                   & 3.67e-3    & 1.49e-3        & 5.94e-4  & 2.37e-4   &9.42e-5   & 0.40 & $\sigma$\\[1ex]
           & $0.6$                   & 4.29e-4    & 1.09e-4        & 2.76e-5  & 6.95e-6   &1.75e-6   & 0.60 & $\sigma$\\[1ex]
           & $0.8$                   & 3.99e-5    & 6.43e-6        & 1.02e-6  & 1.62e-7   &2.57e-8    &0.80 & $\sigma$\\[1ex]
           & $1.2$                   & 1.64e-6    & 1.05e-7        & 6.64e-9  & 4.49e-10  &2.55e-11   &1.20 & $\sigma$\\[1ex]
           & $1.8$                   & 7.58e-7    & 1.21e-9        & 1.92e-11 & 3.05e-13  &4.85e-14   & 1.80 & $\sigma$\\[1ex]
  \hline
\label{1d12}
\end{tabular}
\end{center}
\end{table}

\begin{table}[ht]
\begin{center}
\caption{Maximum errors at $t=1$ and convergence orders in spatial direction with $N=1000$ for Example 1}
\begin{tabular}{cc|cccccc|cc}
     \hline
  $\alpha$    &$\sigma$    $\backslash M$& $8$       &$16$           &$24$      &$32$       &$40$     &$Rate$ & expected order \\[1ex]
       \hline
     0.4   & $0.4$                   & 9.25E-3    & 2.30e-3        & 1.02e-3  & 5.67e-4   &3.60e-4   & 2.02 & $2$\\[1ex]
           & $1.2$                   & 7.68e-3    & 1.92E-3        & 8.51e-4  & 4.79e-4   &3.06e-4   & 2.00 & $2$\\[1ex]
     \hline
   0.6     & $0.6$                   & 8.20e-3    & 2.03e-3        & 8.96e-4  & 4.98e-4   &3.14e-4   & 2.03 & $2$\\[1ex]
           & $0.2$                   & 6.79e-3    & 1.70e-3        & 7.54e-4  & 4.24e-4   &2.71e-4   & 2.00 & $2$\\[1ex]
  \hline
\label{1d13}
\end{tabular}
\end{center}
\end{table}

We investigate the convergence orders in temporal direction by setting $M=1000$ and $t_N=1$.  We show  the maximum errors at $t_N=1$ in Table \ref{1d11}. It can be seen from the results that
when $t$ is far away from $0$, the convergence order tends to $\sigma+1-\alpha$ when $0<\sigma<1$ and  to $2-\alpha$ when $1<\sigma<2$.  We also show the maximum errors as $t\rightarrow 0$ in Table \ref{1d12}. The convergence results indicate that the  convergence order tends to $\sigma$. Then, we investigate the spatial convergence orders by setting $N$=1000 and $t_N=1$. We show the numerical results in Table \ref{1d13}. Clearly, the convergence order tends to $2$. All the numerical results agree with the theoretical findings well.

{\bf Example 2.} ~Consider the nonlinear subdiffusion problems \eqref{i1} with the following conditions
  \begin{eqnarray}
&& (a) ~d=1, f(u)=\sqrt{1+u^2}~\textrm{and}~ u_0(x)=x(1-x),\Omega=(0,1),\nonumber\\
&& (b) ~d=1, f(u)=u-u^3~\textrm{and}~ u_0(x)=\sin(\pi x), \Omega=(0,1),\nonumber \\
&& (c) ~d=2, f(u)=\sqrt{1+u^2}~\textrm{and}~ u_0(x)=\sin(\pi x)\sin(\pi y),\Omega=(0,1)^2,\nonumber\\
&& (d) ~d=2, f(u)=u-u^3~\textrm{and}~ u_0(x)=x(1-x)y(1-y), \Omega=(0,\pi)^2.\nonumber
\end{eqnarray}
One can check that the initial condition $u_0\in H_0^1(\Omega)\cap H^2(\Omega)$. Therefore, we have $\sigma=\alpha$. In the numerical experiments, we take $M=1000$ for one-dimensional problem and $M=10$ for two dimensional problem. The reference solutions are obtained by using small temporal stepsizes. We list the maximum errors at $t=1$ and convergence orders in Table \ref{1d21}. One can see that the convergence order tends to $1$. We also show the maximum errors as $t\rightarrow 0$ and convergence orders in Table \ref{1d22}. Clearly,  the  convergence order tends to $\alpha$. The results further confirm the theoretical results of the present paper.

\begin{table}[ht]
\begin{center}
\caption{Maximum errors at $t=1$ and convergence orders in temporal direction for Example 2}
\begin{tabular}{cc|ccccc|c|ccccccccccc}
     \hline
      $\alpha$        & $case \backslash N$& $10$       &$20$           &$40$      &$80$       &$160$     &$Rate$  & expected order \\[1ex]
     \hline
       $0.4$         &   $(a)$            & 1.74e-4    & 8.36e-5        & 4.04e-5  & 1.94e-5   &8.99e-6   & 1.06 & 1 \\[1ex]
                     &   $(b)$            & 1.49e-3    & 7.14e-4        & 3.46e-4  & 1.66e-4   &7.69e-5   & 1.06 & 1\\[1ex]
                      &  $(c)$            & 9.63e-5    & 4.77e-5        & 2.36e-5  & 1.16e-5   &5.61e-6   & 1.03 & 1\\[1ex]
                     &   $(d)$            & 7.20e-6    & 3.57e-6        & 1.76e-6  & 8.67e-7   &4.19e-7   & 1.02 & 1\\[1ex]
     \hline
       $0.6$         &   $(a)$            & 2.15e-4    & 1.01e-4        & 4.84e-5  & 2.30e-5   &1.06e-5   & 1.08 & 1\\[1ex]
                     &   $(b)$            & 1.88e-3    & 8.84e-4        & 4.23e-4  & 2.01e-4   &9.27e-5   & 1.08 & 1\\[1ex]
                      &  $(c)$            & 1.06e-4    & 5.21e-5        & 2.57e-5  & 1.26e-5   &6.09e-6   & 1.03 & 1\\[1ex]
                     &   $(d)$            & 7.94e-6    & 3.92e-6        & 1.93e-6  & 9.47e-7   &4.57e-7   & 1.03 & 1\\[1ex]
    \hline
       $0.8$         &   $(a)$            & 2.13e-4    & 9.53e-5        & 4.44e-5  & 2.08e-5   &9.48e-6   & 1.10 & 1\\[1ex]
                     &   $(b)$            & 1.96e-3    & 8.74e-4        & 4.07e-5  & 1.90e-5   &8.66e-6   & 1.12 & 1\\[1ex]
                     &   $(c)$            & 7.91e-5    & 3.88e-5        & 1.91e-5  & 9.31e-6   &4.47e-6   & 1.04 & 1\\[1ex]
                     &   $(d)$            & 6.02e-6    & 2.95e-6        & 1.45e-6  & 7.06e-7   &3.40e-7   & 1.03 & 1\\[1ex]
  \hline
\label{1d21}
\end{tabular}
\end{center}
\end{table}

\begin{table}[ht]
\begin{center}
\caption{Maximum errors as $t\rightarrow 0$ and convergence orders in temporal direction for Example 2}
\begin{tabular}{cc|ccccc|c|ccccccccccc}
     \hline
      $\alpha$        & $case \backslash t_N$& $1e-4$       &$1e-5$           &$1e-6$      &$1e-7$       &$1e-8$     &$Rate$ & expected order \\[1ex]
     \hline
       $0.4$         &   $(a)$            & 4.70e-4        & 2.21e-5  & 9.00e-5   &3.61e-5     &1.45e-5   & 0.38   & $\alpha $\\[1ex]
                     &   $(b)$            & 3.65e-3    & 1.87e-3        & 8.36e-4  & 3.50e-4   &1.43e-4   & 0.35   & $\alpha $\\[1ex]
                     &   $(c)$            & 4.96e-3    & 3.05e-3        & 1.47e-3  & 6.34e-4   &2.61e-4   & 0.32   & $\alpha $\\[1ex]
                     &   $(d)$            & 3.31e-4     & 1.87e-4        & 8.26e-5  & 3.39e-5  & 1.37e-5   & 0.35  & $\alpha $\\[1ex]
     \hline
       $0.6$         &   $(a)$            & 1.20e-4        & 3.05e-5  & 7.70e-6   &1.94e-6     &4.87e-7   & 0.60  & $\alpha $\\[1ex]
                     &   $(b)$            & 1.14e-3    & 2.99e-4        & 7.59e-5  & 1.91e-5   &4.81e-6   & 0.59  & $\alpha $\\[1ex]
                     &   $(c)$            & 2.03e-3    & 5.46e-4        & 1.39e-4  & 3.52e-5   &8.84e-6   & 0.59  & $\alpha $\\[1ex]
                     &   $(d)$            & 1.11e-4    & 2.86e-5        & 7.23e-6  & 1.82e-6   &4.57e-7   & 0.60  & $\alpha $\\[1ex]
    \hline
       $0.8$         &   $(a)$            & 1.73e-5    & 2.75e-6        & 4.35e-7  &6.86e-8    &1.09e-8   & 0.80  & $\alpha $\\[1ex]
                     &   $(b)$            & 1.63e-4    & 2.59e-5        & 4.10e-6  & 6.51e-7   &1.03e-8   & 0.80  & $\alpha $\\[1ex]
                     &   $(c)$            & 3.06e-4    & 4.77e-5        & 7.56e-6  & 1.20e-6   &1.90e-7   & 0.80  & $\alpha $\\[1ex]
                     &   $(d)$            & 1.55e-5    & 2.46e-6        & 3.90e-7  & 6.18e-8   &9.80e-9   & 0.80  & $\alpha $\\[1ex]

  \hline
\label{1d22}
\end{tabular}
\end{center}
\end{table}

\section{Conclusions}\label{sec5}
In this paper, we consider the numerical solutions of the fully discrete scheme for solving the nonlinear subdiffusion problems.
The scheme is constructed by using the L1-scheme in temporal directions and central finite difference method for the diffusion term. The pointwise-in-time error estimates  of the L1 sheme are obtained under the regularity parameter $\sigma\in (0,1)\cup(1,2)$ by carefully using the refined discrete fractional-type Gr\"onwall inequality. The present error estimates are sharp and illustrated by several numerical experiments.

\end{document}